\documentclass[12pt]{article}
\usepackage{amsmath}
\usepackage{amsthm}
\usepackage[alphabetic,initials]{amsrefs}
\usepackage{geometry}
\usepackage{verbatim}                
\usepackage[all]{xy}
\usepackage{graphicx}
\usepackage{amssymb}
\usepackage{epstopdf}
\usepackage[utf8]{inputenc}
\DeclareMathOperator{\assoc}{Associator}

\title{$A_\infty$-Algebras Derived from Associative Algebras with a Non-Derivation Differential}
\author{Kaj B\"orjeson}

\date{}                                           

\begin{document}

\maketitle

\begin{abstract}
Given an associative graded algebra equipped with a degree $+1$ differential $\Delta$ we define an $A_\infty$-structure that measures the failure of $\Delta$ to be a derivation. This can be seen as a non-commutative analog of generalized BV-algebras. In that spirit we introduce a notion of associative order for the operator $\Delta$ and prove that it satisfies properties similar to the commutative case. In particular when it has associative order $2$ the new product is a strictly associative product of degree $+1$ and there is a compatibility between the products, similar to ordinary BV-algebras. We consider several examples of structures obtained in this way. In particular we obtain an $A_\infty$-structure on the bar complex of an $A_\infty$-algebra that is strictly associative if the original algebra is strictly associative. We also introduce strictly associative degree $+1$ products for any degree $+1$ action on a graded algebra. Moreover, an $A_\infty$-structure is constructed on the Hochschild cocomplex of an associative algebra with a non-degenerate inner product by using Connes' B-operator.
\end{abstract}



\theoremstyle{plain}
\newtheorem{theorem}{Theorem}
\newtheorem{proposition}{Proposition}
\newtheorem{lemma}{Lemma}
\newtheorem{corollary}{Corollary}

\theoremstyle{definition}
\newtheorem{definition}{Definition}
\newtheorem{example}{Example}

\theoremstyle{remark}
\newtheorem{remark}{Remark}

\section*{Introduction}
Consider a graded commutative algebra equipped with a degree $+1$ differential $\Delta$. There is an $L_\infty$-structure measuring the failure of $\Delta$ to be a derivation. The $L_\infty$-operations are called the Koszul hierarchy, see \cite{Akman} \cite{Koszul}, and are defined as $b_n(a_1,\dots,a_n)=[[[\Delta,L_{a_1}],\dots],L_{a_n}](1)$ where the brackets are commutators of operators and $L_a$ is the operator multiplying by $a$ from the left. Unwrapping this compact definition we see that the first few operations are: $b_1(a)=\Delta(a)$, $b_2(a,b)=\Delta(ab)-\Delta(a)b-(-1)^{|a|}a\Delta(b)$ and $b_3(a,b,c)= \Delta(abc)-\Delta(ab)c-(-1)^{|a|}a\Delta(bc)-(-1)^{(|a|+1)|b|}b\Delta(ac)+\Delta(a)bc+(-1)^{|a|}a\Delta(b)c+(-1)^{|a|+|b|}ab\Delta(c)$. Disregarding signs from element degrees; the operations are a sum over all different ways of applying $\Delta$ to a subset of the inputs with minus signs if there is an odd number of elements outside $\Delta.$ The data of the commutative product and the operator $\Delta$ is called a generalized BV-algebra. Note that no compatibility is assumed and that there is no explicit mention of Lie brackets. If we require that $b_3=0$ we obtain the definition of an ordinary BV-algebra. Put differently, that $b_3=0$ is equivalent to that $\Delta$ is a second order operator or that $b_2$ is a strict degree $+1$ Lie bracket. Saying that $b_{n+1}=0$ is equivalent to $\Delta$ being of $n$:th order. However, if the algebra is not graded commutative the operations $b_n$ does not form an $L_\infty$-structure, see \cite{Bering}. In this note we define a set of operations $m_n$ for an arbitrary graded associative algebra with a degree $+1$ differential $\Delta$. We prove that these operations form an $A_\infty$-structure measuring the failure of $\Delta$ to be a derivation. In analogy with the commutative case we define a notion of \emph{associative order} of the operator $\Delta$ by saying that it has associative order $n$ if $m_{n+1}=0.$ In case $\Delta$ has associative order $2$ or equivalently that $m_3=0$, the operation $m_2$ is a strict degree $+1$ graded associative algebra and it turns out that there is extra compatibility between the products. We consider the combined structure of the two different degree products and $\Delta$ as a non-commutative analog of BV-algebras. Forgetting the operator $\Delta$ yields a non-commutative analog of Gerstenhaber algebras. Note that it does not reduce to the usual notion of BV-algebra in the case the starting algebra is commutative. It should perhaps rather be seen as a BV-algebra in the "associative world" in the sense of \cite{Kontsevich}. 

Any $A_\infty$-algebra determines and is determined by the bar differential on the tensor (co)-algebra of the underlying complex. The bar differential is a coderivation of the coproduct but it is not a derivation with respect to the tensor product. We apply our construction in this case and obtain an $A_\infty$-structure on the tensor module that is strict in case the original $A_\infty$-structure is strict.

 Consider an odd element and a degree $0$ graded associative multiplication on a complex $V$. As an easy example we obtain an degree $+1$ graded associative algebra structure from our construction by letting $\Delta$ be the left multiplication of the odd element. The same construction yields a strict degree $+1$ associative product from any degree $+1$ left action on an associative algebra.

Consider the Hochschild cocomplex of an associative algebra with an invariant non-degenerate bilinear form. The form allows us to move Connes' $B$-operator from the complex to the cocomplex. Applying our construction with this operator and the cup product yields an $A_\infty$-structure. It turns out that $m_2$ is (up to a sign) the Gerstenhaber bracket.

\section*{Conventions}


\begin{definition}
Let $A$ be a graded associative algebra over a field $\mathbb{K}$ of characteristic zero, such that the underlying $\mathbb{Z}$-graded vector space is a complex, that is, it has a degree $+1$ linear operator $\Delta$ such that $\Delta^2=0.$ We call this an algebra with differential. 
\end{definition}

\begin{remark}
Note that we do not require any compatibility between $\Delta$ and the multiplication. The symbol $\Delta$ is chosen to remind of the odd Laplacian operator in a BV-algebra structure.
\end{remark}

\begin{remark}
We work in the symmetric monoidal category of complexes, thus we employ the Koszul sign rule. That is, when we permute homogeneous odd elements we multiple by $-1.$
\end{remark}

\begin{definition}
Let $A$ be an algebra with differential $\Delta$ and product \\${\gamma_2:A^{\otimes 2}\rightarrow A}$. 
Suppose that $\Delta( \gamma_2) - \gamma_2(\Delta\otimes id) -\gamma_2(id \otimes \Delta)=0.$ Then we call $\Delta$ a derivation and $A$ a differential graded algebra.
\end{definition}

\begin{definition}
Suppose $C$ is a graded vector space. Now let $T(C)=\bigoplus_{n=1}^\infty C^{\otimes n}.$ $T(C)$ is an associative algebra with the product given by concatenation of tensor words. It is also a coassociative coalgebra given by the sum over all ways to split a tensor word in two without permuting any elements.
\end{definition}

\begin{definition}
Suppose $A$ is a $\mathbb{Z}$-graded vector space with a degree $+1$ differential $\Delta$. An $A_\infty$-structure on $A$ is a collection $\{a_n\}_{n \geq 2}$ of degree $2-n$ maps $$a_n:A^{\otimes n}\rightarrow A$$ such that the following identity is satisfied for every $n$ (where we put $a_1=\Delta$).

$$\sum_{n=i+j+k\atop k\geq 1, i,j \geq 0}(-1)^{i+jk}a_{i+1+j}(id^{\otimes i}\otimes a_k \otimes id^{\otimes j})=0$$ Equivalently, we can define the structure on the shifted space $A[1]$.

An $A_\infty$-structure on $A[1]$ is a collection $\{m_n\}_{n \geq 2}$ of degree $+1$ maps $$m_n:A^{\otimes n}\rightarrow A$$ such that the following identity is satisfied for every $n$ (where we put $m_1=\Delta$).

$$\sum_{n=i+j+k\atop k\geq 1, i,j \geq 0}m_{i+1+j}(id^{\otimes i}\otimes m_k \otimes id^{\otimes j})=0$$
\end{definition}

\begin{remark}
An $A_\infty$-structure on $A$ where $a_k$ vanishes for $k\neq 2$ is an ordinary graded associative product on $A$.  An $A_\infty$-structure on $A$ where $a_k$ vanishes for $k>2$ is a differential graded algebra. An $A_\infty$-structure on $A[1]$ with $m_k=0$ for $k\neq 2$ is an associative algebra with product of degree $+1$ on $A$. Since the main construction of this note deals with the interplay of products of different degrees we cannot regrade to get rid of the products with odd degree. No matter how we choose it some product will be more complicated. We prefer to construct an $A_\infty$-structure on $A[1]$ to avoid the presence of too complicated signs in the identities we have to check.
\end{remark}

\section*{$A_\infty$-structure from Non-Derivation Differential}
\begin{theorem}
\label{main theorem}
Let $A$ be a graded associative algebra with differential $\Delta$.  Denote multiplication of $n$ ordered elements by the map $\gamma_n:A^{\otimes n} \rightarrow A$.  There is an $A_\infty$-structure on $A[1]$ given by \begin{equation*}m_1=\Delta, \end{equation*} \begin{equation*}m_2= \Delta( \gamma_2) - \gamma_2(\Delta\otimes id) -\gamma_2(id \otimes \Delta), \end{equation*} and  \begin{equation*} m_n=\Delta(\gamma_n)-\gamma_2(\Delta(\gamma_{n-1})\otimes id) -\gamma_2(id\otimes\Delta(\gamma_{n-1}))+\gamma_3(id\otimes \Delta(\gamma_{n-2}) \otimes id),  \end{equation*} for $n\geq 3.$
\end{theorem}

Actually the proof gives a bit more. We have the following a bit more elegant and general result. Looking at the associators of operations $m_n$ can be seen as taking a kind of square. The theorem says that this operation yields the same result as squaring the operator $\Delta$ first. In the case $\Delta^2=0$ it reduces to the previous theorem. This formulation is analogous to a result in the commutative case presented in \cite{Voronov}.

\begin{theorem}
\label{assoc theorem}
Let $A$ be a graded associative algebra with a degree $+1$ operator $\Delta$, not necessarily satisfying $\Delta^2=0$. Denote multiplication of $n$ ordered elements by the map $\gamma_n:A^{\otimes n} \rightarrow A.$ Define maps  \begin{equation}m_{\Delta,1}=\Delta, \label{m1} \end{equation} \begin{equation}m_{\Delta,2}= \Delta( \gamma_2) - \gamma_2(\Delta\otimes id) -\gamma_2(id \otimes \Delta), \label{m2} \end{equation} and  \begin{equation} m_{\Delta,n}=\Delta(\gamma_n)-\gamma_2(\Delta(\gamma_{n-1})\otimes id) -\gamma_2(id\otimes\Delta(\gamma_{n-1}))+\gamma_3(id\otimes \Delta(\gamma_{n-2}) \otimes id), \label{mn} \end{equation} for $n\geq 3.$ Now let $$\assoc_{m_\Delta,n}:= \sum_{n=i+j+k\atop k\geq 1, i,j \geq 0}m_{\Delta,i+1+j}(id^{\otimes i}\otimes m_{\Delta,k} \otimes id^{\otimes j}).$$ Then we have the identity $$\assoc_{m_\Delta,n}=m_{\Delta^2,n}.$$
\end{theorem}

\begin{proof}
For every $n$ we have to check the identity $$\sum_{n=i+j+k\atop k\geq 1, i,j \geq 0}m_{\Delta,i+1+j}(id^{\otimes i}\otimes m_{\Delta,k} \otimes id^{\otimes j})=m_{\Delta^2,n}.$$
Every term is either of the form (Case 1) $$\gamma_{i+1+k+1+m}(id^{\otimes i}\otimes \Delta(\gamma_j) \otimes id^{\otimes k} \otimes\Delta(\gamma_\ell)\otimes id^{\otimes m})$$ or of the form (Case 2)$$\gamma_{i+1+m}(id^{\otimes i}\otimes \Delta(\gamma_{j+1+\ell}(id^{\otimes j}\otimes \Delta(\gamma_k)\otimes id^{\otimes \ell} ))\otimes id^{\otimes m}).$$ 
We will prove the identity by checking that the coefficient in front of every type of term not containing $\Delta^2$ vanishes and that the coefficients of the terms with $\Delta^2$ agree.

\subsection*{Case 1}
{
We look at the coefficient in front of $$\gamma_{i+1+k+1+m}(id^{\otimes i}\otimes \Delta(\gamma_j) \otimes id^{\otimes k} \otimes\Delta(\gamma_\ell)\otimes id^{\otimes m})$$ for fixed $i,j,k,\ell,m.$ In the definition of the product there are no non-zero terms where there are more than one $id$ in front or more than id one $id$ behind $\Delta$. From this we see that the coefficient in front of $\gamma_{i+1+k+1+m}(id^{\otimes i}\otimes \Delta(\gamma_j) \otimes id^{\otimes k} \otimes\Delta(\gamma_\ell)\otimes id^{\otimes m})$ vanishes unless $i,k,m\leq 1.$ 

Therefore it remains to check the following terms: $\gamma_{2}(\Delta(\gamma_j)\otimes\Delta(\gamma_\ell))$, \\ ${\gamma_{3}(id \otimes\Delta(\gamma_j)\otimes\Delta(\gamma_\ell))}$, ${\gamma_{3}(\Delta(\gamma_j) \otimes\Delta(\gamma_\ell)\otimes id)}$, ${\gamma_{3}( \Delta(\gamma_j) \otimes id \otimes\Delta(\gamma_\ell))},$ \\${\gamma_{4}(id\otimes \Delta(\gamma_j) \otimes id \otimes \Delta(\gamma_\ell))},$ ${\gamma_{4}( \Delta(\gamma_j) \otimes id \otimes\Delta(\gamma_\ell)\otimes id)},$ \\${\gamma_{4}(id\otimes \Delta(\gamma_j) \otimes \Delta(\gamma_\ell)\otimes id)},$ and ${\gamma_{5}(id\otimes \Delta(\gamma_j) \otimes id \otimes\Delta(\gamma_\ell)\otimes id)}.$ 

The term $\gamma_{2}(\Delta(\gamma_j)\otimes\Delta(\gamma_\ell))$ has contributions from $m_{\Delta,\ell+1}(m_{\Delta,j},id^{\otimes \ell})$ and \\$m_{\Delta,j+1}(id^{\otimes j},m_{\Delta,\ell})$. They contribute $+1$ and $-1$ respectively; they have different signs because the $\Delta$:s pass each other when calculating one of the terms. 

The term $\gamma_{3}(id \otimes \Delta(\gamma_j) \otimes\Delta(\gamma_\ell))$ has contributions from $m_{\Delta,j+2}(id^{\otimes j+1},m_{\Delta,\ell})$ and $m_{\Delta,\ell+1}(m_{\Delta,j+1},id^{\otimes \ell})$. They contribute with opposite signs; again using the Koszul sign rule. 

The term $\gamma_{3}(\Delta(\gamma_j) \otimes\Delta(\gamma_\ell)\otimes id)$ vanishes similarly.

The term $\gamma_{3}( \Delta(\gamma_j) \otimes id \otimes\Delta(\gamma_\ell)),$ has contributions from $m_{\Delta,\ell+1}(m_{\Delta,j+1},id^{\otimes \ell})$ and $m_{\Delta,j+1}(id^{\otimes j},m_{\Delta,\ell+1})$, again canceling.

The term $\gamma_{4}(id\otimes \Delta(\gamma_j) \otimes id \otimes \Delta(\gamma_\ell))$ has contributions by $m_{\Delta,j+2}(id^{\otimes j+1},m_{\Delta,\ell+1})$ and $m_{\Delta,\ell+1}(m_{\Delta,j+2},id^{\otimes \ell})$. These cancel by the Koszul sign rule.

The term $\gamma_{4}( \Delta(\gamma_j) \otimes id \otimes\Delta(\gamma_\ell)\otimes id)$ vanishes similarly.

The term $\gamma_{4}(id\otimes \Delta(\gamma_j) \otimes \Delta(\gamma_\ell)\otimes id)$ has contributions from $m_{\Delta,j+2}(id^{\otimes j+1},m_{\Delta,\ell+1})$ and $m_{\Delta,\ell+2}(m_{\Delta,j+1},id^{\otimes \ell+1})$, also vanishing.

The term $\gamma_{5}(id\otimes \Delta(\gamma_j) \otimes id \otimes\Delta(\gamma_\ell)\otimes id)$ has contributions from \\$m_{\Delta,\ell+2}(m_{\Delta,j+2},id^{\otimes \ell+1})$ and $m_{\Delta,j+2}(id^{\otimes j+1},m_{\Delta,\ell+2})$ and thus also vanishes by the Koszul rule. }

\subsection*{Case 2}
We now look at the coefficient in front of $$\gamma_{i+1+m}(id^{\otimes i}\otimes \Delta(\gamma_{j+1+\ell}(id^{\otimes j}\otimes \Delta(\gamma_k)\otimes id^{\otimes \ell} ))\otimes id^{\otimes m}),$$ where either $j$ or $\ell$ is non-zero. Changing $i$ or $m$ only multiplies it with $1$, $-1$ or $0$ so it is enough to check the vanishing of coefficient for cases of the form $\Delta(\gamma_{j+1+\ell}(id^{\otimes j}\otimes \Delta(\gamma_k)\otimes id^{\otimes \ell} )).$ 

The term $\Delta(\gamma_{2}(id\otimes \Delta(\gamma_k) ))$ has contributions from $m_{\Delta,1}(m_{\Delta,k+1})$ and from \\$m_{\Delta,2}(id,m_{\Delta,k})$ which vanishes by equation \ref{m1} and \ref{m2}.

The term $\Delta(\gamma_2(\Delta(\gamma_k)\otimes id))$ vanishes similarly.

For $j,\ell \geq 1$ the term $\Delta(\gamma_{j+1+\ell}(id^{\otimes j}\otimes \Delta(\gamma_k)\otimes id^{\otimes \ell} ))$ has contributions from $m_{\Delta,1}(m_{\Delta,k+2})$, $m_{\Delta,2}(id,m_{\Delta,k+1})$, $m_{\Delta,2}(m_{\Delta,k+1},id)$ and from $m_{\Delta,3}(id,m_{\Delta,k},id).$ The sum of the contributing coefficients vanishes by equation \ref{m1}, \ref{m2} and \ref{mn}.

When $j=\ell=0$ it does not necessarily vanish. We want to prove that the coefficient from $$\sum_{n=i+j+k\atop k\geq 1, i,j \geq 0}m_{\Delta,i+1+j}(id^{\otimes i}\otimes m_{\Delta,k} \otimes id^{\otimes j})$$ is the same as the coefficient from $m_{\Delta^2,n}.$ But in this case there is only one contributing term on both sides, the coefficent comes from the definition of $m_{*,i}$ which is the same in both cases.
\end{proof}

\section*{Failure of Being a Derivation and Associative Order of Operators}

\begin{lemma}
\label{order lemma}
Suppose $A$ is a graded associative algebra with a differential $\Delta$ and let $m_n$ be as in Theorem \ref{main theorem}. If $m_n=0,$ then $m_i=0$ for $i>n.$
\end{lemma}

\begin{proof}
Suppose $m_1=\Delta=0,$ then it is clear that all $m_n$ vanishes since all terms use $\Delta.$ Suppose instead that $m_2= \Delta( \gamma_2) - \gamma_2(\Delta\otimes id) -\gamma_2(id \otimes \Delta)=0.$ We want to show that $m_3=\Delta(\gamma_3)-\gamma_2(\Delta(\gamma_{2})\otimes id) -\gamma_2(id\otimes\Delta(\gamma_{2}))+\gamma_3(id\otimes \Delta(id) \otimes id)=0.$ By writing $\Delta(\gamma_3)=\Delta(\gamma_2(\gamma_2\otimes id))$ and using that $m_2=0,$ we see that $m_3$ vanishes. Now suppose $n\geq 4$ and that $m_{n-1}$ vanishes. We want to show that $m_n=\Delta(\gamma_n)-\gamma_2(\Delta(\gamma_{n-1})\otimes id) -\gamma_2(id\otimes\Delta(\gamma_{n-1}))+\gamma_3(id\otimes \Delta(\gamma_{n-2}) \otimes id)=0$. Similarly to the previous case we rewrite $\Delta(\gamma_n)=\Delta(\gamma_{n-1}(id\otimes \gamma_2 \otimes id^{\otimes n-3})) $ and use that $m_{n-1}=0$ to see that $m_n$ also vanishes. Now the lemma follows by induction.
\end{proof}

This lemma motivates the following definition inspired by the commutative case.

\begin{definition}
Suppose $A$ is a graded associative algebra with a differential $\Delta$. We say that $\Delta$ has \textit{associative order} $n$ if $m_{n+1}$ vanishes.
\end{definition}

\begin{remark}
An operator of associative order $1$ is the same thing as a derivation, thus we can look at the $A_\infty$-structure as measuring the failure of $\Delta$ to be a derivation.
\end{remark}


The next theorem shows that there is a compatibility in the case when the operator has associative order $2$. This is analogous to an ordinary BV-algebra and the Gerstenhaber part of it.

\begin{theorem}
\label{compatibility theorem}
Suppose $A$ is a graded associative algebra with multiplication $\gamma_2$ and differential $\Delta$ of associative order $2$.  Then the identities $$\gamma_2(id,m_2)= m_2(\gamma_2,id)$$ and $$\gamma_2(m_2,id)=m_2(id,\gamma_2),$$ hold.
\end{theorem}

\begin{proof}
That $\Delta$ has associative order $2$ is equivalent to the identity $$\Delta(\gamma_3)=\gamma_2(\Delta(\gamma_2),id)+\gamma_2(id,\Delta(\gamma_2))-\gamma_3(id,\Delta(id),id).$$ Now we have $$m_2(\gamma_2,id)=\Delta(\gamma_3)-\gamma_2(\Delta(\gamma_2),id)-\gamma_2(\gamma_2,\Delta(id))=$$ $$= \gamma_2(\Delta(\gamma_2),id)+\gamma_2(id,\Delta(\gamma_2))-\gamma_3(id,\Delta(id),id)-\gamma_2(\Delta(\gamma_2),id)-\gamma_2(\gamma_2,\Delta(id))=$$ $$=\gamma_2(id,\Delta(\gamma_2))-\gamma_3(id,\Delta(id),id)-\gamma_2(\gamma_2,\Delta(id)) = \gamma_2(id,m_2).$$ The other identity is proved in the same way.
\end{proof}

\section*{Triviality on $\Delta$-Cohomology}
Given an $A_\infty$-algebra one has an induced structure on cohomology. The structure from Theorem \ref{main theorem} measures the incompatibility of $\Delta$ with an associative product. Since passing to cohomology kills $\Delta$ one can guess that the induced structure is trivial. This is indeed the case and is analogous to the commutative case.

\begin{theorem}
Let $A$ and $\Delta$ be as in Theorem \ref{main theorem}. The operations $m_k$ are trivial on $\Delta$-cohomology.
\end{theorem}

\begin{proof}
On cohomology every element in the image of $\Delta$ is zero. But every term in the definition of $m_k$ contains images of $\Delta$. 
\end{proof}

\section*{Triangular Matrices and Odd Actions}
As a first very concrete example we consider the algebra of upper triangular matrices.

\begin{example}
Let $A$ be the algebra of upper triangular $2\times 2$-matrices. This has a grading where we consider matrices of the form $\begin{pmatrix} a & 0  \\ 0 & b \end{pmatrix}$ as degree $0$ and matrices of the form $\begin{pmatrix} 0 & c  \\ 0 & 0 \end{pmatrix}$ as degree $1$. Now let us consider the differential given by $\Delta\begin{pmatrix} a & c  \\ 0 & b \end{pmatrix}=\begin{pmatrix} 0 & b  \\ 0 & 0 \end{pmatrix}$. It is easy to check that the multiplication and differential respect the grading. We will determine the structure given by Theorem \ref{main theorem} in this case. By definition $m_1$ is exactly $\Delta.$ Note that the $m_k$:s all raise the degree by one, because of this all multiplications involving elements of degree $1$ will vanish since we have nothing in degree $2$. Thus we only need to compute the $m_k$:s on diagonal matrices. The defining formula gives $$m_2\left(\begin{pmatrix} a & 0  \\ 0 & b \end{pmatrix},\begin{pmatrix} c & 0  \\ 0 & d \end{pmatrix}\right)=\Delta \begin{pmatrix} ac & 0  \\ 0 & bd \end{pmatrix}-\begin{pmatrix} 0 & b  \\ 0 & 0 \end{pmatrix}\begin{pmatrix} c & 0  \\ 0 & d \end{pmatrix}-\begin{pmatrix} a & 0  \\ 0 & b \end{pmatrix}\begin{pmatrix} 0 & d  \\ 0 & 0 \end{pmatrix}$$ $$= \begin{pmatrix} 0 & bd  \\ 0 & 0 \end{pmatrix}-\begin{pmatrix} 0 & bd  \\ 0 & 0 \end{pmatrix}-\begin{pmatrix} 0 & ad  \\ 0 & 0 \end{pmatrix}= \begin{pmatrix} 0 & -ad  \\ 0 & 0 \end{pmatrix}.$$
Similarly to determine $m_3$ we apply the definition to obtain $$m_3\left(\begin{pmatrix} a & 0  \\ 0 & b \end{pmatrix},\begin{pmatrix} c & 0  \\ 0 & d \end{pmatrix},\begin{pmatrix} e & 0  \\ 0 & f \end{pmatrix}\right)=\Delta\begin{pmatrix} ace & 0  \\ 0 & bdf \end{pmatrix} $$ $$-\Delta\left(\begin{pmatrix} ac & 0  \\ 0 & bd \end{pmatrix}\right)\begin{pmatrix} e & 0  \\ 0 & f \end{pmatrix}-\begin{pmatrix} a & 0  \\ 0 & b \end{pmatrix}\Delta\left(\begin{pmatrix} ce & 0  \\ 0 & df \end{pmatrix}\right)+\begin{pmatrix} a & 0  \\ 0 & b \end{pmatrix}\Delta\left(\begin{pmatrix} c & 0  \\ 0 & d \end{pmatrix}\right)\begin{pmatrix} e & 0  \\ 0 & f \end{pmatrix}$$ $$=\begin{pmatrix} 0 & bdf  \\ 0 & 0 \end{pmatrix}-\begin{pmatrix} 0 & bdf  \\ 0 & 0 \end{pmatrix}-\begin{pmatrix} 0 & adf  \\ 0 & 0 \end{pmatrix}+\begin{pmatrix} 0 & adf  \\ 0 & 0 \end{pmatrix}=0.$$
By Lemma \ref{order lemma} we can now see that $m_k$ vanishes for all higher $k$.
\end{example}

\begin{remark}
Note that there is nothing really special about this example except that is small and easily computable. We could have chosen any graded associative algebra with any differential.
\end{remark}

The next example creates a degree $+1$ associative algebra from any graded associative algebra with a choice of degree $+1$ left action (for example left multiplication with a degree $+1$ element).

\begin{example}
Consider a dg algebra $A$ with multiplication $\gamma$ and a degree $+1$ element $\xi$. Denote left multiplication with $\xi$ by $L_\xi.$ It is a degree $+1$ endomorphism of $A$ satisfying $L_\xi(\gamma(a,b))=\gamma(L_\xi(a),b).$ Consider the construction in Theorem \ref{assoc theorem} applied to $\Delta=L_{\xi}.$ We see that $m_2(a,b)= (-1)^{|a|} \gamma(a,L_\xi(b))=(-1)^{|a|} a\xi b$ and that $m_3(a,b,c)=L_\xi^2(abc)-L_\xi^2(ab)c-aL_\xi^2(bc)+aL_\xi^2(b)c= \xi^2abc-\xi^2abc-a\xi^2bc+a\xi^2bc=0.$ Now consider an arbitrary degree $+1$ endomorphism $\mathcal{L}$ satisfying $\mathcal{L}(\gamma(a,b))=\gamma(\mathcal{L}(a),b).$ This also induces a strict degree $+1$ associative algebra structure since $m_3$ vanishes by applying $\mathcal{L}(\gamma(a,b))=\gamma(\mathcal{L}(a),b).$
\end{example}

\section*{$A_\infty$-structure on the Bar Complex}
An alternative characterization of $A_\infty$-structure is the following.
\begin{theorem}
An $A_\infty$-structure on $A$ is equivalent to a degree $+1$ square-zero coderivation of the reduced tensor coalgebra on $A[1]$. An $A_\infty$-structure $\{m_n\}$ correspond to the coderivation $$\Delta = \sum_{i=1}^n\sum_{j=0}^{n-i}id^{\otimes j}\otimes m_i \otimes id^{\otimes n-i-j}.$$
\end{theorem}

\begin{proof}
See for example \cite{LodayVallette}.
\end{proof}

This coderivation is however not necessarily a derivation with respect to the tensor product, enabling us to apply Theorem \ref{main theorem}.

\begin{theorem}
Given an $A_\infty$-structure $\{m_n\}$ on $A[1]$ there is an $A_\infty$-structure $\{t_i\}$ on the shifted reduced tensor algebra $T(A)[1]$. We have $t_1(v_1\otimes \dots \otimes v_n) =$ $$\sum_{i=1}^n\sum_{j=0}^{n-i}(-1)^{|v_1|+\dots+|v_j|}v_1\otimes \dots v_j \otimes m_i(v_{j+1}\otimes \dots \otimes v_{j+i}) \otimes v_{j+i+1}\otimes \dots \otimes v_n,$$ $t_2(v_1\otimes \dots \otimes v_n\bigotimes w_1\otimes \dots \otimes w_m)=$ $$\sum_{i<n\atop j <m}(-1)^{|v_1|+\dots+|v_i|} v_1\otimes \dots v_i \otimes m_{n+m-i-j}(v_{i+1}\otimes \dots \otimes w_{m-j}) \otimes w_{m-j+1}\otimes \dots \otimes w_m $$ and for $k\geq 3$ we have $$t_k(u_{11}\otimes \dots \otimes u_{1n_1}\bigotimes  \dots \bigotimes u_{k1}\otimes \dots \otimes u_{kn_k})=\sum_{i<n_1\atop j <n_k} (-1)^{|u_{11}|+\dots+|u_{1i}|}$$ $$u_{11}\otimes \dots u_{1i} \otimes m_{\# inputs -i-j}(u_{1(i+1)}\otimes \dots \otimes u_{k(n_k-j)}) \otimes u_{k(n_k-j+1)}\otimes \dots \otimes u_{kn_k}.$$
\end{theorem}

\begin{proof}
By definition $t_1$ is just application of $\Delta$, the signs originate from the Koszul sign rule. The definition of $t_2$ is $$t_2(v_1\otimes \dots \otimes v_n\bigotimes w_1\otimes \dots \otimes w_m)=\Delta(v_1\otimes \dots \otimes v_n\bigotimes w_1\otimes \dots \otimes w_m)-$$ $$\Delta(v_1\otimes \dots \otimes v_n)\bigotimes w_1\otimes \dots \otimes w_m-(-1)^{|v_1|+\dots+|v_n|}v_1\otimes \dots \otimes v_n\bigotimes\Delta( w_1\otimes \dots \otimes w_m).$$ By the definition of $\Delta$ we see that the first term consists of all the way we can contract using the multiplication and the other terms are contractions using only elements from the first respectively the second tensor word. The remaining terms are thus the contractions involving elements from both words. This can be written as $t_2(v_1\otimes \dots \otimes v_n\bigotimes w_1\otimes \dots \otimes w_m)=$ $$\sum_{i<n\atop j <m}(-1)^{|v_1|+\dots+|v_i|} v_1\otimes \dots v_i \otimes m_{n+m-i-j}(v_{i+1}\otimes \dots \otimes w_{m-j}) \otimes w_{m-j+1}\otimes \dots \otimes w_m.$$ The definition of $t_k$ for $k\geq 3$ is $t_k(u_{11}\otimes \dots \otimes u_{1n_1}\bigotimes  \dots \bigotimes u_{k1}\otimes \dots \otimes u_{kn_k})=$ $$\Delta(u_{11}\otimes \dots \otimes u_{1n_1}\bigotimes\dots \bigotimes u_{k1}\otimes \dots \otimes u_{kn_k})-$$ $$\Delta(u_{11}\otimes \dots \otimes u_{1n_1}\bigotimes \dots)\bigotimes u_{k1}\otimes \dots \otimes u_{kn_k}$$ $$-(-1)^{|u_{11}|+\dots+|u_{1n_1}|}u_{11}\otimes \dots \otimes u_{1n_1}\bigotimes\Delta(\dots \bigotimes u_{k1}\otimes \dots \otimes u_{kn_k})$$ $$+(-1)^{|u_{11}|+\dots+|u_{1n_1}|}u_{11}\otimes \dots \otimes u_{1n_1}\bigotimes\Delta(\dots )\bigotimes u_{k1}\otimes \dots \otimes u_{kn_k}.$$ Similarly to the previous case we see that the first term corresponds to any contractions, the second term corresponds to contractions not involving the last tensor word, the third corresponds to contractions not involving the first word and the fourth to contractions not involving the first or last word. Thus we see that the remaining terms are only the contractions involving elements both from the first and the last tensor word. Alternatively this can be written as $$t_k(u_{11}\otimes \dots \otimes u_{1n_1}\bigotimes  \dots \bigotimes u_{k1}\otimes \dots \otimes u_{kn_k})=\sum_{i<n_1\atop j <n_k} (-1)^{|u_{11}|+\dots+|u_{1i}|}$$ $$u_{11}\otimes \dots u_{1i} \otimes m_{\# inputs -i-j}(u_{1(i+1)}\otimes \dots \otimes u_{k(n_k-j)}) \otimes u_{k(n_k-j+1)}\otimes \dots \otimes u_{kn_k}.$$
\end{proof}

\begin{remark}
Note that $t_1$ is the differential of the bar construction. If $m_i=0$ for $i\geq  3$ we see that $t_i$ vanishes for $i\geq 3$ making the bar resolution complex into a dg algebra (with odd degree). In this case the product is easily described by $t_2(v_1\otimes \dots \otimes v_n\bigotimes w_1\otimes \dots \otimes w_m)=$ $$ (-1)^{|v_1|+\dots+|v_{n-1}|} v_1\otimes \dots v_{n-1} \otimes m_2(v_{n} \otimes w_1) \otimes w_{2}\otimes \dots \otimes w_m$$ and Theorem \ref{compatibility theorem} says that there is a compatibility with the tensor product (which is also easy to prove directly).
\end{remark}

\section*{Hochschild Cocomplex and the Dualized Connes' $B$-Operator}
For an introduction to Hochschild cohomology, see for example \cite{Loday}. Consider a finite-dimensional associative unital algebra $A$ with a symmetric, invariant non-degenerate inner product $<,>.$ In \cite{Tradler}, a degree $-1$ differential $\Delta$ is considered on Hochschild cochains $C^\bullet(A,A)$ defined by $$<\Delta f(a_1,\dots,a_{n-1}),a_n>=\sum_{i=1}^n(-1)^{i(n-1)}<f(a_i,\dots,a_{n-1},a_n,a_1,\dots,a_{i-1}),1>.$$ This operator is Connes' $B$-operator transferred from chains to cochains by using the inner product. The following is shown in \cite{Tradler}.

\begin{theorem}
\label{tradler}
$\Delta$ is a chain map with respect to the Hochschild differential and $\Delta^2=0.$ Furthermore, on Hochschild cohomology the following identity holds $$[a,b]=-(-1)^{(|a|-1)|b|+1}\left(\Delta(a\smile b) -\Delta(a)\smile b -(-1)^{|a|} a\smile \Delta(b) \right),$$ where $[,]$ is the Gerstenhaber bracket  and $\smile$ is the cup product as defined in \cite{Gerstenhaber}.
\end{theorem}

\begin{proof}
See \cite{Tradler}.
\end{proof}

It is clear that the construction in Theorem \ref{main theorem} works if we reverse gradings and in that case we obtain a homologically graded $A_\infty$-structure. We can therefore apply that machinery to $C^\bullet(A,A)$ equipped with the cup product and the differential $\Delta.$ Doing this we obtain an $A_\infty$-structure on the chain level.

\begin{theorem}
There is a homologically graded $A_\infty$-structure on $C^\bullet(A,A)[-1]$ such that it induces an $A_\infty$-structure given by maps $\{m_n\}_{n\geq1}$ on the cohomology $HH^\bullet(A,A)[-1]$ such that $m_1=\Delta$ and $m_2$ is the Gerstenhaber bracket up to a sign.
\end{theorem}

\begin{proof}
The $A_\infty$-structure is built from the operator $\Delta$ and the cup product $\smile$. Since both are compatible with the Hochschild coboundary $\delta$, the $A_\infty$-structure is well defined on the cohomology. By definition $m_1$ is given by $\Delta$ and by unwinding the definition of $m_2$ we see that Theorem \ref{tradler} shows that $m_2$ is the Gerstenhaber bracket up to a sign.
\end{proof}

\section*{Acknowledgements}
I would like to thank J. Alm, A. Berglund, T. Backman and S. Merkulov for interesting discussions. A special thanks to S. Merkulov for many useful suggestions and comments during the preparation of this paper.

~\\

\noindent\textsc{Department of Mathematics \\ Stockholm University \\ 106 91 Stockholm, Sweden \\
\emph{E-mail address:} {\bf kaj@math.su.se}}

\end{document}